\documentclass{amsart}
\usepackage{amsmath,amssymb}
\newtheorem{theorem}{Theorem}[section]
\newtheorem{lemma}[theorem]{Lemma}

\newtheorem{claim}[theorem]{Claim}
\newtheorem{definition}[theorem]{Definition}
\newtheorem{proposition}[theorem]{Proposition}

\newtheorem{corollary}[theorem]{Corollary}

\newtheorem{remark}[theorem]{Remark}
\newtheorem{problem}[theorem]{Problem}

\def\er{\mathbb R}

\def\C{\mathcal C}

\def\H{\mathcal H}

\def\A{\mathcal A}

\def\cO{\mathcal O}

\newcommand{\graph}{\operatorname{graph}}

\newcommand{\intr}{\operatorname{int}}
\newcommand{\nor}{\operatorname{nor}}

\newcommand{\spt}{\operatorname{spt}}
\newcommand{\dist}{\operatorname{dist}}

\newcommand{\ep}{\varepsilon}

\newcommand{\eps}{\varepsilon}

\newcommand{\loc}{\operatorname{loc}}

\DeclareMathOperator*{\esslim}{ess\,lim}
\newcommand{\reach}{\operatorname{reach}}
\newcommand{\sph}{S^{d-1}}

\newcommand{\llc}{\;\halfsq\;}
\newcommand{\lrc}{\;\ihalfsq\;}

\def\halfsq{\hbox{\kern1pt\vrule height 7pt\vrule width6pt height 0.4pt depth0pt\kern1pt}}
\def\ihalfsq{\hbox{\kern1pt \vrule width6pt height 0.4pt depth0pt
                   \vrule height 7pt \kern1pt}}

\def\eqn#1$$#2$${\begin{equation}\label#1#2\end{equation}}

\begin{document}
\title[Normal cycles of sets with d.c.\ boundary]{Normal cycles and curvature measures of sets with d.c.\ boundary}
\author{Du\v san Pokorn\'y \and Jan Rataj}
\thanks{The authors were supported by a cooperation grant of the Czech and the German science foundation, GA\v CR project no.\ P201/10/J039}
\address{Charles University, Faculty of Mathematics and Physics, Sokolovsk\'a 83, 18675 Pra\-ha~8, Czech Republic}
\email{dpokorny@karlin.mff.cuni.cz, rataj@karlin.mff.cuni.cz}

\begin{abstract}
We show that for every compact domain in a Euclidean space with d.c. (delta-convex) boundary there exists a unique Legendrian cycle such that the associated curvature measures fulfil a local version of the Gauss-Bonnet formula. This was known in dimensions two and three and was open in higher dimensions. In fact, we show this property for a larger class of sets including also lower-dimensional sets. We also describe the local index function of the Legendrian cycles and we show that the associated curvature measures fulfill the Crofton formula.
\end{abstract}

\maketitle
\section{Introduction}

The goal of extending the notion of curvature to non-smooth sets (with singularities) belongs to important tasks of geometry for decades. We consider here only subsets of the Euclidean space $\er^d$, though some approaches can be transferred to the Riemannian setting. It turned out that curvature measures can be derived from a more complex structure called normal cycle; this idea can be found by Sulanke \& Wintgen \cite{SW} for smooth sets, Z\"ahle \cite{Z86} for sets with positive reach, Fu \cite{Fu1} for more general sets, and later developed by others.

To describe the basic idea, consider a full-dimensional compact subset $A$ of $\er^d$ with $C^2$-smooth boundary, and let $\nor A$ be its unit normal bundle, i.e., $\nor A$ consists of pairs $(x,n)$, where $x$ is a boundary point of $A$ and $n$ is the unit outer normal vector to $A$ at $x$. The normal cycle $N_A$ of $A$ is the $(d-1)$-dimensional current which is given by integrating over the oriented manifold $\nor A$, i.e.,
$$N_A(\phi)=\int_{\nor A}\phi=\int_{\nor A}\langle\xi_A,\phi\rangle\, d\H^{d-1}$$
for any smooth $(d-1)$-form $\phi$ on $\er^{2d}$ (here $\xi_A$ is a prescribed unit simple $(d-1)$-vectorfield orienting $\nor A$ and $\H^{d-1}$ denotes the $(d-1)$-dimensional Hausdorff measure).

Given $k\in\{ 0,\ldots,d-1\}$, let $\varphi_k$ be the $k$th Lipschitz-Killing differential $(d-1)$-form on $\er^{2d}$ which can be described by
\begin{eqnarray*}
\lefteqn{\langle a^1\wedge\cdots\wedge a^{d-1},\varphi_k(x,n)\rangle}\\
&=&\cO_{d-k-1}^{-1}\sum_{\sum_i\sigma(i)=d-1-k}\langle\pi_{\sigma(1)}a^1\wedge\cdots\wedge\pi_{\sigma(d-1)}a^{d-1}\wedge n,\Omega_d\rangle,
\end{eqnarray*}
where $a^i$ are vectors from $\er^{2d}$, $\pi_0(x,n)=x$ and $\pi_1(x,n)=n$ are coordinate projections, the sum is taken over finite sequences $\sigma$ of values from $\{0,1\}$, $\Omega_d$ denotes the volume form in $\er^d$ and $\cO_{d-1}=\H^{d-1}(\sph)=2\pi^{d/2}/\Gamma(\frac d2)$. 
Note that, in particular, $\varphi_0$ is an $\cO_{d-1}^{-1}$-multiple of the $\pi_1$-pull-back of the volume form $n\lrc\Omega_d$ on $\sph$:
$$\varphi_0=\cO_{d-1}^{-1}\pi_1^{\#}(n\lrc\Omega_d).$$

Integrating $\varphi_k$ over $\nor A$ yields the $k$th (total) curvature of $A$, which can also be expressed as the integral of the $k$th symmetric function of principal curvatures of $A$:
$$N_A(\varphi_k)=\int_{\nor A}\varphi_k=C_k(A).$$
For completeness, we define also $C_d(A)=\H^d(A)$.
The $k$th curvature measure of $A$, $C_k(A,\cdot)$, is obtained by localizing with a Borel set $F\subset\er^d$
$$C_k(A,F)=(N_A\llc(F\times\er^d))(\varphi_k)=\int_{\nor A}({\bf 1}_F\circ\pi_0)\varphi_k.$$

In case of sets with singularities, the normal direction need not be determined uniquely. A useful and well tractable set class containing both smooth sets and closed convex sets is the family of sets with positive reach (i.e., sets for which any point within a certain distance apart has its unique nearest point in the set). Federer \cite{Fe59} introduced curvature measures for sets with positive reach by means of a local Steiner formula, and Z\"ahle \cite{Z86} defined normal cycles for these sets.

Fu \cite{Fu1} observed that the normal cycle of a set has a tangential property called later Legendrian, and he called {\it Legendrian cycle} any closed rectifiable $(d-1)$-dimensional current in $\er^d\times\sph$ with this property (see Section~\ref{Legendrian} for exact definition). Fu also showed that the restriction of a Legendrian cycle $T$ to the Gauss curvatures form $\varphi_0$, $T\llc\varphi_0$, determines $T$ uniquely. Later \cite{Fu94} he introduced a condition on the Legendrian cycle forcing the validity of the Gauss-Bonnet formula, not only in the global version ($C_0(A,\er^d)=\chi(A)$), but also for the set $A$ intersected with halfspaces, for almost all halfspaces of $\er^d$. He showed that subanalytic sets admit such Legendrian cycles. We formulate an equivalent condition in Definition~\ref{D-NC} and call a Legendrian cycle {\it normal cycle} if this condition is satisfied.

The particular case of (full-dimensional) sets whose boundary can be represented locally as graph of a Lipschitz function (Lipschitz domain for short) was treated in \cite{RZ05}. Of course, an additional condition has to be imposed, in order that the total boundary curvature is bounded. The normal cycle was obtained by approximation with parallel sets. In a recent paper \cite{Fu2} Fu showed that if a function is {\it strongly approximable} (i.e., can be approximated well by $C^2$-smooth functions in certain sense concerning second derivatives, see Definition~\ref{strong_approx}) then it admits a second order Taylor expansion almost everywhere (which is a property close the existence of a normal cycle for the subgraph). In particular, every strongly approximable function is Monge-Amp\`ere, see Definition~\ref{MA}. Fu asked whether, in particular, delta-convex (d.c.) functions (differences of two convex functions) are strongly approximable. This is easy to see for functions of one variable and known for functions of two variables, see \cite{Fu00}. 

In Section~\ref{S_WDC} we introduce WDC sets as sublevel sets of d.c.\ functions at weakly regular values. These can be considered as natural generalizations of sets with positive reach which can be equivalently characterized as sublevel sets of semiconcave functions at weakly regular values (see \cite{Kl}). The class of locally WDC sets contain full-dimensional domains with d.c.\ boundary, as well as lower-dimensional d.c.\ surfaces or transversal unions of sets with positive reach.

We present in Section~\ref{SA} an answer the problem formulated in slightly different wording in \cite[Problem~7.3, p.~458]{AA}, \cite{Fu1} and \cite{Fu2}:

\begin{theorem} \label{T1}
Every d.c.\ function on $\er^d$ is strongly approximable and, hence, also Monge-Amp\`ere.
\end{theorem}

This rather simple observation is the cornerstone of the paper. The proof is based on a formula for determinants (Lemma~\ref{KL}) which makes it possible to find upper bounds for minors of differences of two matrices by means of those of convex combinations. 

After that, we apply the theory of auras due to Fu (\cite{Fu94}) to d.c.\ functions and we show:

\begin{theorem} \label{T2}
Any compact WDC set in $\er^d$ admits a normal cycle (in the sense of Definition~\ref{D-NC}).
\end{theorem}

In particular, curvature measures can be introduced for WDC sets. An important step in the proof is the fact that the set of tangent hyperplanes to the graph of a d.c.\ function has measure zero, which was shown by Pavlica and Zaj\' i\v cek \cite{PZ} (their result is an application of some duality ideas to a deep result on directions of line segments on the boundary of a convex body by Ewald, Larman and Rogers \cite{ELR}).

The normal cycles of WDC sets fulfil (by definition) the Gauss-Bonnet formula for intersection with almost all halfspaces. We also show a local formula for the index function of the normal cycle which imply the additivity (with respect to unions and intersections). 

Finally, we show in Section~\ref{S-CF} the Crofton formula for WDC sets. The set $\A^d_m$ of affine $m$-subspaces of $\er^d$ with invariant measure $\mu^d_m$ are introduced in Section~\ref{S-prel}.

\begin{theorem}[Crofton formula] \label{T3}
Let $A$ be a compact WDC set. Then, for any integers $0\leq k\leq m\leq d$,
$$\int_{\A^d_m}C_k(A\cap E)\,\mu^d_m(dE)=\beta^d_{d+k-m,m}C_{d+k-m}(A),$$
where
$$\beta^d_{i,j}=\frac{\Gamma(\tfrac{i+1}2)\Gamma(\tfrac{j+1}2)}{\Gamma(\tfrac{d+1}2)\Gamma(\tfrac{i+j-d+1}{2})}.$$
\end{theorem}

Again, we use the fact that the set of hyperplanes tangent to the graph of a d.c.\ function has $d$-dimensional measure zero.

One of the problems when constructing the normal cycle for WDC sets is that we do not know whether the corresponding unit normal bundle is a rectifiable set. The positive result on rectifiability would imply, for example, the validity of the principal kinematic formula. It seems that rectifiability of the unit normal bundle is equivalent to the rectifiability of the set of segments on the boundary of a convex body studied in \cite{ELR} and this appears to be an interesting open problem. 

Analysing the proof of the Principal kinematic formula in \cite[Corollary~2.2.2]{Fu94}, one sees that a weaker property than rectifiability suffices. Let $A,B$ be two compact WDC sets with d.c.\ auras $f,g$, respectively, and let $\nor(f,0), \nor(g,0)$ be the two unit normal bundles (see Definition~\ref{aura}). We conjecture that $\H^{2d-1}(\nor(f,0)\times\nor(g,0))=0$. This would already imply the validity of the Principal kinematic formula for $A$ and $B$.

\section{Preliminaries}  \label{S-prel}

The basic setting will be the $d$-dimensional Euclidean space $\er^d$ with scalar product ``$\cdot$'' and norm $|\cdot|$, $\sph$ stands for the unit sphere. By $B(x,r)$ we denote the closed ball of centre $x$ and radius $r$. The $k$-dimensional Hausdorff measure will be denoted by $\H^k$.

If $v\in\sph$ and $t\in\er$, let $H_{v,t}$ denote the halfspace $\{ y\in\er^d:\, y\cdot v\leq t\}$.
$\H^{d}\llc(\sph\times\er)$ induces a natural measure on the family of halfspaces (through the mapping $(v,t)\mapsto H_{v,t}$) and under $H_{v,t}\mapsto\partial H_{v,t}$ we obtain a natural measure on the family of hyperplanes in $\er^d$. These measures should be always understood when speaking about ``almost all halfspaces'' or ``almost all hyperplanes''. 

Let $G(d,i)$ be the Grassmannian of $i$-dimensional linear subspaces of $\er^d$ with unique invariant probability measure $\nu_i^d$, and let ${\mathcal A}_i^d$ denote the set of all $i$-dimen\-si\-onal affine subspaces of $\er^d$ with invariant measure $\mu_i^d$ given by
$$\mu_i^d(U)=\int_{G(d,i)}\lambda^{d-i}\{z\in L^\perp:\, L+z\in U\}\, \nu_i^d(dL)$$
for any Borel subset $U$ of ${\mathcal A}_i^d$.

We shall use the fact that the measure $\mu_i^d$ can equivalently be given as
\begin{equation}  \label{int_geom}
\mu_i^d(U)=\int_{{\mathcal A}_{j}^d}\int_{{\mathcal A}_i^{j}(E)}
\mu_i^{j}(U\cap{\mathcal A}_i^{j}(E))\, \mu_{j}^d(dE)
\end{equation}
whenever $i\leq j\leq d-1$, 
where ${\mathcal A}_i^{j}(E)$ is the set of all affine $i$-subspaces of $E\in{\mathcal A}_{j}^d$
This is a well-known fact from integral geometry and follows e.g.\ from \cite[Theorem~7.1.2]{SnWe}.

Let $U\subset\er^d$ be open and $f:U\to\er$ locally Lipschitz. By Rademacher's theorem, the differential $df(x)$ (and gradient $\nabla f(x)$) exists at $\H^d$-almost all $x\in U$. At any $x\in U$, the (Clarke) subdifferential of $f$ at $x$, $\partial^*f(x)$, is defined to be the closed convex hull of the set of all accumulation points of gradients of $f$ at regular points converging to $x$. Clearly, the graph of the subdifferential 
$$\graph\partial^*f = \{ (x,u):\, x\in U,\, u\in\partial^*f(x)\}$$
is a closed subset of $\er^d\times\er^d$.

\section{Delta convex functions and WDC sets}  \label{S_WDC}

A real function $f$ defined on a convex set is called d.c.\ (or delta-convex) when it can be expressed as a difference of two convex (concave) functions.
A function $f$ defined on an open set $U$ is said to be locally d.c., if for every $x\in U$ there is a convex set $K_x\subset U$ such that $f|_{K_{x}}$ is d.c.
Note that every d.c.\ function is locally Lipschitz and that every semi-convex (or semi-concave) function is also d.c.
A mapping $F:\er^d\to\er^k$ is called a {\it d.c.\ mapping} if every component of $F$ is a d.c.\ function.

It is well known that for two d.c.\ functions $f,g$, not only $f+g$, but also $fg$, $\max(f,g),$ $\min(f,g)$ are d.c.
Also, if $F,G$ are two d.c. mappings and $F\circ G$ makes sense then $F\circ G$ is a d.c. mapping as well.
For more details see \cite{VZ} or \cite{Har}.

Let $f$ be a Lipschitz function on $\er^d.$ A real number $c$ is called a {\it weakly regular value} if there is an $\eps>0$ such that for every $c<f(x)<c+\eps$ and every $v\in\partial^*f(x)$ the inequality $|v|\geq\eps$ holds. 

\begin{definition}\label{WDC} \rm
A compact set $A\subset\er^d$ is called {\it WDC} ({\it weakly delta-convex}) if there is a d.c.\ function $f:\er^d\to\er$ with a weakly regular value $c$ such that $A=f^{-1}((-\infty,c])$.
\end{definition}

\begin{remark} \rm
\begin{enumerate}
\item[(i)] Any compact set $A\subset\er^d$ with positive reach \cite{Fe59} is WDC. Indeed, the distance function $d_A(x)=\dist(x,A)$ is semiconcave (cf.\   \cite[Satz~(2.8)]{Kl}), hence, also d.c., and $0$ is a weakly regular value of $d_A$ since $d_A$ has unit gradient at all points $x$ with $0<d_A(x)<\reach(A)$, see \cite[Theorem~4.8]{Fe59}. In particular, compact convex sets are WDC.
\item[(ii)] If a compact set $A\subset\er^d$ is represented as a finite union of sets with positive reach which intersect transversally (see \cite{Z87}) then $A$ is WDC. This follows from the fact that, given two sets $A,B$ with positive reach which intersect transversally, their distance functions $d_A,d_B$ are nondegenerate auras (see Definition~\ref{aura}) that do not touch and Theorem~\ref{T_additivity} implies that $A\cup B$ is WDC.
\item[(iii)]
If $\Psi$ is a $\C^2$ diffeomorphism and $A$ a WDC set with witnessing function $f$, then
$\Psi(A)$ is WDC as well, with witnessing function $f\circ\Psi^{-1}$ (see also Theorem~\ref{T-diff}).
\end{enumerate}
\end{remark}
We call a set $A\subset\er^d$ {\it locally WDC} if for every $x\in A$ there is $U_x$, an open neighborhood of $x$, and a WDC set $A_x$ such that $A\cap U_x=A_x\cap U_x.$
It is not clear whether a locally WDC set is immediately WDC (cf. Problem~\ref{locwdc}), but due to the local character of the notion "being normal cycle", we can prove that compact locally WDC sets admit normal cycles just from the fact that compact WDC sets do (see Theorem~\ref{locth}).

\subsection{Delta-convex domains and surfaces}
A {\it d.c.\ domain} in $\er^d$ is a set which can be locally represented as subgraph of a d.c.\ function.
Following \cite[Definiton 2.11]{RZaj} we will call a set $M\subset\er^d$ a {\it $k$-dimensional d.c. surface} if for every $x\in M$ there is an open set $U_x$, a subspace $A_x\in G(d,k)$ and a d.c.\ mapping $\phi_x:A_x\to A_x^{\bot}$ such that 
$$
M\cap U_x=\{u+\phi_x(u):u\in A_x\}\cap U_x.
$$

\begin{proposition}
\begin{enumerate}
\item[{\rm (i)}] Each d.c.\ domain in $\er^d$ is a locally WDC set.
\item[{\rm (ii)}] Every $k$-dimensional d.c.\ surface in $\er^d$ is a locally WDC set.
\end{enumerate}
\end{proposition}

\begin{proof}
Let $E$ be a linear subspace of dimension $d-1$ in $\er^d$ and let $f:E\to\er$ be a d.c. function. 
Without any loss of generality we can suppose that $f$ is Lipschitz.
Let $e\in S^{d-1}$ be a fixed vector orthogonal to $E$. 
To prove (i) we need to prove that the set $A:=\{u+ve:u\in E,\,v\leq f(u)\}$ is locally WDC.
Consider the function $h:\er^d\to\er$ defined as 
$$
h(x):=\max\left(0,x\cdot e-f(x-(x\cdot e)e)\right).
$$
Then $h$ is d.c. and also $h(x)=0$ if and only if $x\in A$.
Moreover, the directional derivative of $h$ in the direction $e$ is equal to $1$ and therefore (using Lipschitzness of $f$) we obtain that $0$ is a weakly regular value of $h.$

Part (ii) can be proved similarly. 
\end{proof}

\section{Legendrian and normal cycles}  \label{Legendrian}
We follow the notation and terminology from the Federer's book \cite{Fe69}.
Given an open subset $U$ of $\er^d$ (or, more generally, of a $d$-dimensional smooth submanifold of a Euclidean space) and $0\leq k\leq d$ an integer, let ${\bf I}_k(U)$ denote the space of $k$-dimensional integer multiplicity rectifiable currents in $U$. Each current $T\in{\bf I}_k(U)$ can be represented by integration as
\begin{equation} \label{current}
T=(\H^{k}\llc W(T))\wedge \iota_Ta_T,
\end{equation}
where $W(T)$ is a $(\H^{k},k)$-rectifiable subset of $U$ (``carrier'' of $T$), $a_T$ is a unit simple tangent $k$-vectorfield of $W(T)$ and $\iota_T$ is an integer-valued integrable function over $W(T)$ (``index function'') associated with $T$. Of course, the carrier $W(T)$ is not uniquely determined and need not be closed, in contrast with the support $\spt T$ which is closed by definition, but it is not clear whether it retains the rectifiability property.

The mass norm ${\bf M}(T)$ of a current $T$ is defined as the supremum of values $T(\phi)$ over all differential forms $\phi$ with $|\phi|\leq 1$. For a compact set $K\subset U$, the flat seminorm ${\bf F}_K(T)$ of $T$ is the supremum of $T(\phi)$ over all forms $\phi$ with $\spt\phi\subset K$, $|\phi|\leq 1$ and $|d\phi|\leq 1$.
The topology of flat convergence, i.e., convergence in flat seminorms ${\bf F}_K$ (denoted $T=(F)\lim_i T_i$) is often used for currents, see \cite{RZ01}. This implies the weak convergence ($T_i(\phi)\to T(\phi)$ for any smooth form $\phi$) and is, in fact, equivalent to it if $T_i,T$ are cycles with uniformly bounded mass norms, see \cite[Theorem~8.2.1]{KP08}. 

\begin{definition} \rm
A {\it Legendrian cycle} is an integer multiplicity rectifiable $(d-1)$-current $T\in{\bf I}_{d-1}(\er^d\times S^{d-1})$ with the properties:
\begin{eqnarray}
&&\partial T=0\quad (T\text{ is a cycle}),\\
&&T\llc\alpha=0\quad (T\text{ is Legendrian}),
\end{eqnarray}
where $\alpha$ is the contact $1$-form in $\er^d$ acting as 
$\langle (u,v),\alpha(x,n)\rangle=u\cdot n$ (cf.\ \cite{Fu94}). 
\end{definition}

By the fundamental uniqueness theorem due to Fu \cite[Theorem~4.1]{Fu1}, a compactly supported Legendrian current $T$ is uniquely determined by its restriction to the Gauss form $T\llc\varphi_0$ and this can be described as follows. For any $\phi\in C^\infty_c(\er^d\times\sph)$,
\begin{eqnarray}   
T(\phi\varphi_0)&=&\cO_{d-1}^{-1}\int_{\sph}\sum_{x\in\er^d}\phi(x,w)\iota_T(x,w)\, \H^{d-1}(dw)\quad  \label{Tg}  \\
&=&\sum_{x\in\er^d}\phi(x,w)\iota_T(x,n)\quad \text{for a.a.\ }w\in\sph , \label{Tg2}
\end{eqnarray}
where $\iota_T$ is the {\it index function} of $T$ (a locally $\H^{d-1}$-integrable integer-valued function on $:\er^d\times\sph$), cf.\ \cite[Proposition~4, Theorem~4]{RZ05}.

For the definition of the slice $\langle T,\pi_1,v\rangle$ of $T$ by the coordinate projection $\pi_1:(x,n)\mapsto n$ at point $v$, see \cite[Section~4.3]{Fe69}, in view of \S4.3.13. (We apply the $(d-1)$-form $n\lrc\Omega_d$ on $\sph$ defining the orientation.) The slice
$\langle T,\pi_1,v\rangle\in{\bf I}_0(\er^d\times\sph)$ is a zero-dimensional integral current (hence, a signed counting measure) for $\H^{d-1}$-almost all $v\in\sph$. The following description of the slice follows from \eqref{Tg}:
$$\langle T,\pi_1,v\rangle = (\H^0\llc (\pi_1)^{-1}\{v\})\wedge \iota_T,$$
hence,
\begin{equation}  \label{E_slice}
\langle T,\pi_1,v\rangle (\phi)=\sum_x\iota_T(x,v)\phi(x,v),\quad \phi\in C^0_c(\er^d\times\sph),
\end{equation}
for $\H^{d-1}$-almost all $v\in\sph$.

Let $v\in\sph$ and $t\in\er$ be given. We shall say that the current $T$ {\it touches} the halfspace $H_{v,t}$ (or, equivalently, that $H_{v,t}$ touches $T$) if there exists a point $x\in\er^d$ such that $(x,-v)\in\spt T$ and $x\cdot v=t$. 

Let $\chi$ denote the Euler-Poincar\'e characteristic.

\begin{definition}[Fu]  \label{D-NC}  \rm
We say that a compact set $A\subset\er^d$ admits a {\it normal cycle} $T$ if $T$ is a Legendrian cycle satisfying \begin{equation}  \label{th}
\text{almost all halfspaces do not touch }T
\end{equation} 
and
\begin{equation}  \label{E-NC}
\langle T,\pi_1,-v\rangle (H_{v,t}\times\sph)=\chi (A\cap H_{v,t})\text{ for }\H^d\text{-almost all }(v,t)\in S^{d-1}\times\er.
\end{equation}
Such a $T$ is then unique (see Remark~\ref{rem-uniq}), we write $T=N_A$ and call it the {\it normal cycle associated with} $A$.
\end{definition}

\begin{remark} \rm  \label{rem-uniq}
\begin{enumerate}
\item[(i)] The uniqueness of $T$ in \eqref{E-NC} follows from \cite[Theorem~3.2]{Fu94} and Lemma~\ref{L_equiv} below.
\item[(ii)] There are various classes of sets known to admit a normal cycle, as (compact) sets with positive reach \cite{Z86}, ${\mathcal U}_{\operatorname{PR}}$ sets \cite{RZ01} or subanalytic sets \cite{Fu94}.
\end{enumerate}
\end{remark}

Given a compactly supported Legendrian cycle $T$ satisfying \eqref{th},
the $(d-1)$-current ${\mathcal J}(T,v,t)$ was defined in \cite[p.~145]{RZ05} for almost all $(v,t)\in S^{d-1}\times\er$ (see also \cite[Theorem~3.1]{Fu94} where the current ${\mathcal I}(T,v,t)$ differs only by a constant multiple). Note that if, in particular, $T=N_A$ is the normal cycle of a compact set $A$ with positive reach, then ${\mathcal J}(N_A,v,t)$ is the restriction of the normal cycle of $A\cap H_{v,t}$. Assuming that $H_{v,t}$ does not touch $T$, we get from the definition
\begin{equation}  \label{def_J}
{\mathcal J}(T,v,t)=T\llc(\intr H_{v,t}\times\sph)+S
\end{equation}
with a current $S$ with $\spt S\subset\partial H_{v,t}\times(\sph\setminus\{-v\})$. By compactness, there exists a $\delta>0$ (possibly depending on $v,t$) such that
\begin{equation} \label{delta}
\spt S\subset\partial H_{v,t}\times\{ w\in\sph:\, v\cdot w\geq -1+\delta\}.
\end{equation}
Fu established in \cite{Fu94} the condition
\begin{equation}   \label{LGB}
{\mathcal J}(T,v,t)(\varphi_0)=\chi(A\cap H_{v,t})\text{ for }\H^d\text{-almost all }(v,t)\in S^{d-1}\times\er
\end{equation} 
relating a Legendrian cycle $T$ to a compact set $A\subset\er^d$ as a ``local Gauss-Bonnet formula''. In fact, under \eqref{th}, this condition is equivalent to \eqref{E-NC}, which is shown in the following lemma due to Joseph Fu.

\begin{lemma}  \label{L_equiv}
Let $T$ be a compactly supported Legendrian cycle satisfying \eqref{th}. Then, conditions \eqref{E-NC} and \eqref{LGB} are equivalent.
\end{lemma}

\begin{proof}
Applying \eqref{E_slice}, we obtain that 
\begin{equation}  
\langle T,\pi_1,-v\rangle (H_{v,t}\times\sph)=\sum_{x:\, x\cdot v\leq t}\iota_T(x,-v)
\end{equation}
for almost all $(v,t)\in\sph\times\er$.

Let $(v,t)$ be such that $T$ does not touch $H_{v,t}$ and ${\mathcal J}(T,v,t)$ is defined. Since ${\mathcal J}(T,v,t)$ is a Legendrian cycle, we get from \eqref{Tg2}
$${\mathcal J}(T,v,t)(\varphi_0)=\sum_x\iota_{{\mathcal J}(T,v,t)}(x,w)$$
for almost all $w\in\sph$. If $\delta=\delta(v,t)>0$ is such that \eqref{delta} holds then, due to \eqref{def_J},
$${\mathcal J}(T,v,t)(\varphi_0)=\sum_{x:\, x\cdot v<t}\iota_T(x,w)$$
for almost all $w\in\sph\cap\{ w:\, w\cdot v<-1+\delta\}$. 
Note that, for any $w\in\sph\cap\{ w:\, w\cdot v<-1+\delta\}$ for which the right-hand side is determined (which is the case up to a $\H^{d-1}$ zero set), the right-hand side of the last equation (and, hence, also the left-hand side) is constant on some neighbourhood of $(v,t)\in\sph\times\er$ whenever $H_{v,t}$ does not touch $T$. A standard argument yields that we can match $v$ and $w$ choosing $w=-v$ almost everywhere, i.e.,
$${\mathcal J}(T,v,t)(\varphi_0)=\sum_{x:\, x\cdot v<t}\iota_T(x,-v)=\sum_{x:\, x\cdot v\leq t}\iota_T(x,-v)$$
for almost all $(v,t)\in\sph\times\er$. A comparison with \eqref{E_slice} completes the proof.
\end{proof}

\begin{remark} \rm
It is often convenient to replace \eqref{E-NC} or \eqref{LGB} with a condition prescribing a particular form of the index function $\iota_T$ of $T$. Fu \cite{Fu1} considered the following natural form 
\begin{equation} \label{index}
\iota_T(x,n):=\lim_{r\searrow0}\lim_{s\searrow0}[\chi(A\cap B(x,r)\cap\{p:(p-x)\cdot n\leq t\})|^{t=s}_{t=-s}].
\end{equation}
We obtain in Proposition~\ref{P_local} a slightly different expression where the limits will be replaced by certain weaker forms.
\end{remark}

\section{Monge-Amp\`ere functions and auras}

We will start with the definition of Monge-Amp\`ere functions. 

\begin{definition}[Fu \cite{Fu1}]  \label{MA}  \rm
Let $f:\er^d\to\er$ be locally Lipschitz. 
We say that $f$ is {\it Monge-Amp\`ere} if there exists a (necessarily unique) cycle $[df]\in{\bf I}_d(\er^d\times\er^d)$ such that
\begin{enumerate}
\item[(i)] $[df]$ is Lagrangian, i.e., $[df]\llc\omega=0$, where $\omega$ denotes the symplectic $2$-form in $\er^d\times\er^d$ defined as $\omega=d\alpha$ ($\omega(x,y)=\sum_idx_i\wedge dy_i$ in canonical coordinates),
\item[(ii)] $\pi_0|\spt [df]$ is proper,
\item[(iii)] for any function $\phi\in C^\infty_c(\er^d\times\er^d)$,
$$[df](\phi\cdot(\pi_0)^{\#}\Omega_d)=\int_{\er^d}\phi(u,\nabla f(u))\,\H^{d}(du)$$
(note that the gradient $\nabla f$ is defined $\H^d$-almost everywhere since $f$ is locally Lipschitz).
\end{enumerate}
\end{definition}

\begin{remark} \label{prop}  \rm
\begin{enumerate}
\item Jerrard \cite{Jerr} extended the definition to functions from the Sobolev space $W^{1,1}_{\loc}$. For our purpose, the Lipschitz setting is sufficient.
\item By \cite[Theorem~2.2]{Fu1}, the support of $[df]$ is contained in the graph of the subdifferential of $f$:
\begin{equation} \label{subdif}
\spt [df]\subset \graph \partial^*f.
\end{equation}
\end{enumerate}
\end{remark}

\begin{definition}[{Fu \cite[\S1.1.1]{Fu94}}] \label{aura}\rm
Let $A\subset \er^d$ be compact. An {\it aura} for $A$ is a proper Monge-Amp\`ere function $f:\er^d\to[0,\infty)$ such that $f^{-1}\{0\}=A$. We say that the aura $f$ for $A$ is {\it nondegenerate} if there exists an open set $U\subset\er^d$ containing $A$ and such that
$$\inf\{ |u|:\, x\in U\setminus A,\,u\in\partial^*f(x)\}>0.$$
\end{definition}

Let us denote the mapping
$$\nu: (x,y)\mapsto (x,y/|y|),\quad (x,y)\in \er^d\times(\er^d\setminus\{0\}).$$
If $f$ is a nondegenerate aura for $A$, the support of $[df]\llc(U\setminus A)$ is contained in the domain of $\nu$. Define
\begin{eqnarray*}
\nor(f,0)&:=&\nu(\overline{\graph (\partial^*f|U\setminus A)}\cap (A\times\er^d)),\\
N(f,0)&:=&\nu_{\#}\Big(-\partial ([df]\llc \pi_0^{-1}(U\setminus A))\llc\pi_0^{-1}(U)\Big).
\end{eqnarray*}
Note that none of $\nor(f,0)$, $N(f,0)$ depend on the choice of the neighbourhood $U$, and that
$$\nor(f,0)\subset \nu\Big(\graph(\partial^*f|\partial A)\Big).$$
We shall call $\nor(f,0)$ {\it unit normal bundle} of $f$.

For almost all $0<r<r_0:=\inf\{ f(x):\, x\in \er^d\setminus U\}$ we can define the slices
$$N(f,r):=\nu_{\#}\langle [df],\pi_0\circ f,r\rangle.$$
The following results were shown in \cite[\S1.1]{Fu94}.

\begin{proposition}  \label{P_aura}
If $A\subset\er^d$ is compact and $f$ a nondegenerate aura for $A$ then
\begin{enumerate}
\item $N(f,0)$ is a Legendrian cycle in $\er^d\times\sph$,
\item $\spt N(f,0)\subset \nor(f,0)\subset\partial A\times\sph$,
\item $N(f,0)=(F)\esslim_{r\to 0_+}N(f,r)$,
\item $N(f,0)(\varphi_0)=\chi(A)$.
\end{enumerate}
\end{proposition}

\section{Strong approximability of delta-convex functions}  \label{SA}

\begin{definition}[Fu \cite{Fu2}]  \label{strong_approx}  \rm
Let $U\subset\er^{d}$ be open and $f:U\to\er.$
We say that $f$ is {\it strongly approximable} if there is a sequence $f_{1},f_{2},...\in C^{2}(U)$ that converges to $f$ in $L_{\loc}^1(U)$ and such that
$$
\int_{K}\biggl|\det\biggl(\frac{\partial^{2}f_{k}}{\partial x_{i}\partial x_{j}}\biggr)_{i\in I,j\in J}\biggr|\leq C(f,K)
$$
for every $K\subset U$ compact and $I,J\subset\{1,\ldots,d\}$ of the same cardinality. (If $I=J=\emptyset$, the determinant should be understood as $1$.) 
\end{definition}

It is known that every locally Lipschitz and strongly approximable function is Monge-Amp\`ere and that every concave function on $\er^{d}$ is strongly approximable and therefore Monge-Amp\`ere.
(For the proofs of these facts and other information, see \cite{Fu1}.)

Fu \cite{Fu1} asked whether all d.c.\ functions are strongly approximable (and therefore Monge-Amp\`ere).
The positive answer is based on the following simple formula.

\begin{lemma}  \label{KL}
Let $A,B$ be matrices $d\times d.$
Then 
\begin{equation}\label{det}
\det(A-B)=\frac{1}{d!}\sum_{k=0}^{d}(-1)^{k}\binom{d}{k}\det((d-k)A+kB).
\end{equation}
\end{lemma}

\begin{proof}
The formula is a direct consequence of the formula for computing finite differences (see \cite[p. 9]{Jor})
\footnote{This fact was told to us independently by Noam D. Elkies, Darij Grinberg and Jan Mal\'y and made it possible to reduce our original proof from the first version of the paper.}
$$
\Delta^m [P](t)=\sum_{k=0}^{m}(-1)^{k}\binom{m}{k}P(t+m-k),
$$
Indeed, if we consider the polynomial $\tilde P(t)= \det(dB+t(A-B))$, then
$$
\Delta^d [\tilde P](0)=\sum_{k=0}^{d}(-1)^{k}\binom{d}{k}\tilde P(d-k)=\sum_{k=0}^{d}(-1)^{k}\binom{d}{k}\det((d-k)A+kB).
$$
On the other hand, since $\tilde P$ has degree at most $d$, the $d$th difference is in fact a constant equal to $d!a_d$, where $a_d$ is a coefficient of $\tilde P$ corresponding to $t^d$. 
To finish the proof it is now sufficient to observe that $a_d=\det (A-B).$
\end{proof}

Using formula (\ref{det}) we get the following corollary.

\begin{corollary}\label{vector}
Let $U\subset\er^{d}$ be open and $f,h$ real functions on $U$.
Suppose that every convex combination of $f$ and $g$ is strongly approximable.
Then every linear combination of $f$ and $g$ is strongly approximable as well.
\end{corollary}

\begin{proof}
First note that if every convex combination of $f$ and $g$ is strongly approximable, then every linear combination with positive coefficients is strongly approximable as well.
Suppose that $f_{k}\to f$ and $g_{k}\to g$ are the sequences guaranteed by the strong approximability.
It suffices to prove that $f-ag$ is strongly approximable for $a>0$.
We have $f_{k}-ag_{k}\to f-ag$ in $L_{\loc}^1(U)$. 
Choose $I,J\subset\{1,\dots,d\}$ with $|I|=|J|=m$. 
Then we can write
$$
\begin{aligned}
\int_{K}\biggl|&\det\biggl(\frac{\partial^{2}(f_{k}-ag_{k})}{\partial x_{i}\partial x_{j}}\biggr)_{i\in I,j\in J}\biggr|\\
&=\int_{K}\biggl|\frac{1}{m!}\sum_{l=0}^{m}(-1)^{l}\binom{m}{l}\det\biggl(\frac{\partial^{2}((m-l)f_{k}+lag_{k})}{\partial x_{i}\partial x_{j}}\biggr)_{i\in I,j\in J}\biggr|\\
&\leq\frac{1}{m!}\sum_{l=0}^{m}\binom{m}{l}\int_{K}\biggl|\det\biggl(\frac{\partial^{2}((m-l)f_{k}+lag_{k})}{\partial x_{i}\partial x_{j}}\biggr)_{i\in I,j\in J}\biggr|\\
&\leq\frac{1}{m!}\sum_{l=0}^{m}\binom{m}{l}C((m-l)f+lag,K),<\infty
\end{aligned}
$$
where $C((m-l)f+lag,K)$ are the constants from Definition~\ref{strong_approx}).
\end{proof}

\begin{proof}[Proof of Theorem~\ref{T1}]
Follows directly from Corollary~\ref{vector}.
\end{proof}

\begin{corollary} \label{WDC-aura}
Any compact WDC set admits a nondegenerate d.c.\ aura. 
\end{corollary}

\begin{proof}
Indeed, if $A=f^{-1}((-\infty,c])$ with a d.c.\ function $f:\er^d\to\er$ and weakly regular value $c$, then $\tilde{f}(x)=\max(0,f(x)-c)$ is a nondegenerate d.c.\ aura for $A$.
\end{proof}

\section{Normal cycles of WDC sets}

We know from the last section that a compact WDC set $A\subset\er^d$ admits a nondegenerate aura $f$ which is a d.c. function. This implies already that there exists a Legendrian cycle $N(f,0)$ with support within $\partial A\times\sph$ fulfilling the Gauss-Bonnet formula $N(f,0)(\varphi_0)=\chi(A)$, see Proposition~\ref{P_aura}. In order to verify that 
$$N_A:=N(f,0)$$ 
is the normal cycle of $A$, it remains to show \eqref{E-NC}.

For the whole section, let $A\subset\er^d$ be a compact WDC set with nondegenerate d.c.\ aura $f$.
Denote
$$E_f=\{(v,x\cdot v)\in\sph\times\er:\,(x,-v)\in \nor(f,0)\}.$$
If we associate elements of $E_f$ with halfspaces of $\er^d$, $E_f$ contains all halfspaces touching the current $N(f,0)$, in the sense of the definition given in Section~\ref{Legendrian}, due to Proposition~\ref{P_aura}~(b).

\begin{proposition}  \label{P_th}
If $f$ is a nondegenerate d.c. aura of a compact set $A\subset\er^d$ then $\H^d(E_f)=0$.
\end{proposition}

\begin{proof}
Denote
$$E^{d+1}(f)=\Big\{(1+|u|^2)^{-1/2}\Big( (u,-1),x\cdot u-f(x)\Big):\, x\in\er^d,\, u\in\partial^*f(x)\Big\}.$$
Note that $E^{d+1}(f)\subset S^d\times\er$ corresponds to the family of halfspaces of $\er^{d+1}$ touching $\graph f$ (i.e., with boundary hyperplanes tangent to $\graph f$ in the Clarke sense). A result of Pavlica and Zaj\'\i\v cek \cite[Theorem~4.3]{PZ} says that $\H^{d+1}(E^{d+1}(f))=0$. (In fact, Theorem~4.3 in \cite{PZ} is stated only for tangent hyperplanes in smooth points, but an analysis of the proof shows that the set $E(g,h)$ from \cite[Lemma~4.1]{PZ} contains all tangent hyperplanes in the Clarke sense since it is closed and its fibres with fixed first component are convex.)

Note that if $x\in A$ and $u\in\partial^*f(x)$ then, since $f$ attains its minimum at $x$, also $\alpha u\in\partial^*f(x)$ for any $\alpha\in [0,1]$.

From the nondegeneracy of $f$, there exists $\delta>0$ such that if $(v,t)\in\nor(f,0)$ then there exists $x\in\er^d$ with $x\cdot v=t$ and $s\delta v\in\partial^*f(x)$ for all $s\in[0,1]$. Consequently, the mapping
$$h:(v,t,s)\mapsto (s^2\delta^2+1)^{-1/2}\left((s\delta v,-1),s\delta t\right)$$
embeds $E_f\times[0,1]$ into $E^{d+1}(f)$. Since $h$ is locally bilipschitz, the result of Pavlica and Zaj\'\i\v cek implies that $\H^{d+1}(E_f\times[0,1])=0$. Applying the Fubini theorem, we get $\H^d(E_f)=0$.
\end{proof}

Given $(v,t)\in\sph\times\er$, we define the function
$$g_{v,t}:x\mapsto\max\{ v\cdot x,0\},\quad x\in\er^d.$$
It is not difficult to see the following fact.

\begin{lemma} \label{L_aura_cap_H}
Let $f$ be a d.c.\ nondegenerate aura for a compact set $A\subset\er^d$ and let $(v,t)\in (\sph\times\er)\setminus E_f$. Then $f+g_{v,t}$ is a nondegenerate d.c.\ aura for $A\cap H_{v,t}$ and, consequently,
$$N(f+g_{v,t},0)(\varphi_0)=\chi(A\cap H_{v,t}).$$
\end{lemma}

Lemma~\ref{L_aura_cap_H} is a special case of the following result.
We shall say that two nondegenerate auras $f,g$ in $\er^d$ {\it touch} if there exists a pair $(x,n)\in\nor(f,0)$ such that $(x,-n)\in\nor(g,0)$.

\begin{proposition}  \label{P_int_aur}
Let $f,g$ be nondegenerate d.c.\ auras for compact sets $A,B\subset\er^d$, respectively. If $f$ and $g$ do not touch then $f+g$ is a nondegenerate d.c.\ aura for $A\cap B$.
\end{proposition}

\begin{proof}
Obviously, $f+g$ is a d.c.\ aura for $A\cap B$. We shall show that it is nondegenerate.

Let $\ep>0$ and $U,V$ be open neighbourhoods of $A,B$, respectively, such that $|u|\geq \ep$ and $|v|\geq\ep$ whenever $x\in U\setminus A$, $y\in V\setminus B$, $u\in\partial^*f(x)$ and $v\in\partial^*g(x)$. Since $f$ and $g$ do not touch we know that if $x\in\partial A\cap\partial B$, $(x,m)\in \nor(f,0)$ and $(x,n)\in \nor(g,0)$ then the angle formed by the vectors $m,n$ is less than $\pi$. Using the closedness of $\graph\partial^*f$ and $\graph\partial^*g$ and compactness of $A,B$, we find by a standard argument that there exists $\delta>0$ and an open neighbourhood $W$ of $\partial A\cap\partial B$ such that 
$u\cdot v\geq (-1+\delta)|u||v|$ whenever $x\in W$, $u\in\partial^*f(x)$ and $v\in\partial^*g(x)$. We have now an open cover
$$A\cap B\subset G:=(U\cap\intr B)\cup(V\cap\intr A)\cup W$$
of $A\cap B$. Take $x\in G\setminus (A\cap B)$ and $w\in\partial^*(f+g)(x)$. If $x\in (U\cap\intr B)\setminus A$ then $\partial^*g=\{0\}$, $w\in\partial^*f(x)$ and $|w|\geq\ep$ by the nondegeneracy of $f$. Analogously, if $x\in (V\cap\intr A)\setminus B$ then $\partial^*f(x)=\{0\}$, $w\in\partial^*g(x)$ and $|w|\geq\ep$ by the nondegeneracy of $g$. If, finally, $x\in W\setminus (A\cap B)$ then $w=u+v$ for some $u\in\partial^*f(x)$ and $v\in\partial^*g(x)$ and we get
$$|w|=|u+v|=\sqrt{|u|^2+|v|^2+2u\cdot v}\geq \sqrt{|u|^2+|v|^2+2|u||v|(-1+\delta)}.$$
Since $x\not\in A\cap B$, at least one of the vectors $u,v$ has norm at least $\ep$, say $|u|\geq\ep$. Then
$$|w|\geq \sqrt{\ep^2+|v|^2+2\ep|v|(-1+\delta)}\geq\ep\sqrt{2\delta-\delta^2}\geq\ep\sqrt{\delta}$$
if $\delta<1$. Thus, the aura $f+g$ is nondegenerate and the proof is complete.
\end{proof}

In order to verify \eqref{E-NC}, we need the following relation.

\begin{lemma}  \label{L_tech}
Let $f$ be a d.c.\ nondegenerate aura for a compact set $A\subset\er^d$. Then for $\H^d$-almost all $(v,t)\in\sph\times\er$,
\begin{equation}  \label{E_chi}
\chi(A\cap H_{v,t})=\sum_{x:\, x\cdot v<t}\iota_{N(f,0)}(x,-v)
\end{equation}
and
\begin{equation}  \label{E_sect}
\langle N(f,0),\pi_1,-v\rangle(H_{v,t}\times\sph)=N(f+g_{v,t},0)(\varphi_0).
\end{equation}
\end{lemma}

\begin{proof}
First, note that since $g_{v,t}=0$ on $\intr H_{v,t}$, we have
$$N(f+g_{v,t},0)\llc\intr H_{v,t}=N(f,0)\llc\intr H_{v,t}.$$
Assume that $H_{v,t}\not\in E_f$ (which is true for $H^{d}$-almost all $(v,t)$ by Proposition~\ref{P_th}). Then, by compactness, there exists $\delta>0$ such that $v\cdot w\geq -1+\delta$ whenever $(x,w)\in\nor(f,0)$ and $x\cdot v=t$. Using the relation between $\partial^*f$ and $\partial^*(f+g_{v,t})$, we get the following:
$$(x,w)\in\nor(f+g_{v,t},0),\ x\cdot v=t\implies w\cdot v\geq -1+\delta.$$
$N(f+g_{v,t},0)$ is a Legendrian cycle (see Proposition~\ref{P_aura}). Thus, we can proceed as in the proof of Lemma~\ref{L_equiv} and get
$$N(f+g_{v,t},0)(\varphi_0)=\sum_x\iota_{N(f+g_{v,t},0)}(x,w)$$
for a.a.\ $w\in\sph$. Taking into account the considerations from the beginning of the proof, we obtain
$$N(f+g_{v,t},0)(\varphi_0)=\sum_{x:\ x\cdot v\leq t}\iota_{N(f,0)}(x,w)$$
for almost all $v,w\in\sph$ such that $H_{v,t}\not\in E_f$ and $w\cdot v< -1+\delta$. Again by compactness, the right hand side does not change if we perturb $(v,t)$ slightly. This implies (cf.\ the proof of Lemma~\ref{L_equiv}) that
$$N(f+g_{v,t},0)(\varphi_0)=\sum_{x:\ x\cdot v\leq t}\iota_{N(f,0)}(x,-v)=\sum_{x:\ x\cdot v<t}\iota_{N(f,0)}(x,-v)$$
for almost all $v,w\in\sph$ and, applying Lemma~\ref{L_aura_cap_H}, we get \eqref{E_chi}.

Also $N(f,0)$ is a Legendrian cycle and the procedure used in the proof of Lemma~\ref{L_equiv} yields
$$\langle N(f,0),\pi_1,-v\rangle (H_{v,t}\times\sph)=\sum_{x:\, x\cdot v\leq t}\iota_{N(f,0)}(x,-v)$$
for a.a.\ $v\in\sph$, and a comparison with the last but one formula proves \eqref{E_sect}.
\end{proof}

\begin{proof}[Proof of Theorem~\ref{T2}]
Let $A$ be a compact WDC set in $\er^d$. By Corollary~\ref{WDC-aura}, $A$ admits a nondegenerate d.c.\ aura $f$. Combining Lemma~\ref{L_aura_cap_H} and \eqref{E_sect}, we obtain \eqref{E-NC}, verifying that $N_A=N(f,0)$ is the normal cycle of $A$.
\end{proof}

We describe now the relation between the normal cycles of a WDC set and its diffeomorphic image. If $L:\er^d\to\er^d$ is linear we denote its adjoint (transpose) map by $L^*$.

\begin{theorem}   \label{T-diff}
Let $A\subset\er^d$ be compact and WDC and let $\Psi:\er^d\to\er^d$ be a ${\mathcal C}^2$ diffeomorphism. Then, $\Psi(A)$ is WDC as well and
$$N_{\Psi(A)}=\hat{\Psi}_{\#}N_A,$$
where $\hat{\Psi}(x,n)=(\Psi(x),(d\Psi(x)^*)^{-1}(n)/|(d\Psi(x)^*)^{-1}(n)|)$.
\end{theorem}

\begin{proof}
Denote $\tilde{\Psi}(x,y)=(\Psi(x),(d\Psi(x)^*)^{-1}y)$ and note that $\hat{\Psi}=\nu\circ\tilde{\Psi}$. If $f$ is differentiable at $x$ and $y=\Psi(x)$ we have by the chain rule
$$\nabla(f\circ\Psi^{-1})(y)=(d\Psi(x)^*)^{-1}(\nabla f(x)).$$
We see from the definitions that if $f$ is a nondegenerate d.c.\ aura for $A$ then $f\circ\Psi^{-1}$ is a nondegenerate d.c.\ aura for $\Psi(A)$, and that
$$[d(f\circ\Psi^{-1})]=\tilde{\Psi}_{\#}[df].$$
Further, if $U$ is an open neighbourhood of $A$ such that the subgradient of $f$ is nondegenerate on $U\setminus A$, then $\Phi(U)$ is an open neighbourhood of $\Phi(A)$ with nondegenerate subgradient of $f\circ \Psi^{-1}$ on $\Psi(U)\setminus \Psi(A)=\Psi(U\setminus A)$, and we have from the definition
\begin{eqnarray*}
N(f\circ\Psi^{-1},0)&=&
\nu_{\#}\left( -\partial(\tilde{\Psi}_{\#}[df]\llc\pi_0^{-1}(\Psi(U\setminus A)))\llc\pi_0^{-1}(\Psi(U))\right)\\
&=&
\nu_{\#}\left( -\partial(\tilde{\Psi}_{\#}([df]\llc\pi_0^{-1}(U\setminus A))\llc\pi_0^{-1}(\Psi(U))\right)\\
&=&
\nu_{\#}\tilde{\Psi}_{\#}\left( -\partial([df]\llc\pi_0^{-1}(U\setminus A))\llc\pi_0^{-1}(U)\right)\\
&=&\hat{\Psi}_{\#}N(f,0),
\end{eqnarray*}
verifying the assertion.
\end{proof}

\section{Local description of the normal cycle}

Let $K\subset\er^d$ be a convex body (i.e., a nonempty, compact and convex set). Clearly, the distance function 
$$d_K:\, x\mapsto\dist(X,k)$$
is a nondegenerate aura for $K$.

\begin{lemma}  \label{L_AK}
Let $A\subset\er^d$ be a compact WCD set and let $K$ be a convex body in $\er^d$. If some nondegenerate d.c.\ aura $f$ of $A$ does not touch $d_K$ then $A\cap K$ is WDC and
$$N_{A\cap K}\llc\intr K=N_A\llc\intr K.$$
\end{lemma}

\begin{proof}
$f+d_K$ is a nondegenerate aura for $A\cap K$ by Proposition~\ref{P_int_aur}. Thus, $A\cap K$ is WDC.
Since $f+d_K=f$ on $\intr K$, the equality $N(f+d_K,0)\llc\intr K=N(f,0)\llc\intr K$ follows.
\end{proof}

\begin{definition} \rm
A family $\mathcal V$ of compact subsets of $\er^d$ is a {\it Vitali system} if for any $x\in \er^d$ and $\delta>0$ there exists $K\in\mathcal V$ such that $x\in\intr K$ and $K\subset B(x,\delta)$. If $f$ is a function defined on $\mathcal V$ and $a\in\er$, we write $\lim_{K\to x}f(K)=a$ if for any $\eps>0$ there exists $\delta>0$ such that $|f(K)-a|<\eps$ whenever $x\in \intr K$ and $K\subset B(0,\delta)$.
\end{definition}

Let $f$ be a nondegenerate d.c.\ aura for $A\subset\er^d$ compact. Let $\mathcal V$ denote the system of all convex bodies in $\er^d$ such that $d_K$ does not touch $f$.

\begin{proposition}  \label{Vitali}
$\mathcal V$ is a Vitali system.
\end{proposition}

The proof will follow from two auxiliary lemmas.

Let $f$ be a nondegenerate d.c.\ aura.
We shall say that an affine subspace $F\in{\mathcal A}_i^d$ is {\it tangent} to $f$ if there exists a pair $(x,n)\in\nor(f,0)$ such that $x\in F$ and $n\perp F$. Let $T_i(f)$ denote the set of all tangent affine $i$-subspaces to $f$.

\begin{lemma}\label{nulovamira}
$\mu_i^d(T_i(f))=0$ for all $i=1,\ldots, d-1$.
\end{lemma}

\begin{proof}
For $i=d-1$, the assertion is given in Proposition~\ref{P_th}. For general $i\leq d-1$, we shall proceed by induction over $d$. For $d=2$ there is nothing to prove. If $d>2$ we apply the decomposition \eqref{int_geom}:
\begin{eqnarray*}
\mu_i^d(T_i(f))&=&\int_{T_{d-1}(f)}\int_{{\mathcal A}_i^{d-1}(E)}
\mu_i^{d-1}(T_i(f)\cap{\mathcal A}_i^{d-1}(E))\, \mu_{d-1}^d(dE)\\
&+&\int_{{\mathcal A}_{d-1}^d\setminus T_{d-1}(f)}\int_{{\mathcal A}_i^{d-1}(E)}
\mu_i^{d-1}(T_i(f)\cap{\mathcal A}_i^{d-1}(E))\, \mu_{d-1}^d(dE).
\end{eqnarray*}
The first summand vanishes since $\mu_{d-1}^d(T_{d-1}(f))=0$.
If, on the other hand, $E\in\A_{d-1}^d$ is not tangent to $f$ then it is easy to see that the restriction $f|E$ is a nondegenerate d.c.\ aura in the subspace $E$ and, by the induction assumption, $\mu_i^{d-1}(T_i(f)\cap\A_i^{d-1}(E))=\mu_i^{d-1}(T_i(M|E))=0$. Hence, the second summand vanishes as well, and the proof is finished.
\end{proof}

\begin{lemma}\label{vektory}
Suppose that $A_i\subset \A_{i}^{d}$ are such that $\nu_i^d(\A_{i}^{d}\setminus A_i)=0$ for every $i=0,...,d-1.$
Then there are linearly independent directions $v_1,...,v_d\in S^{d-1}$ and sets $E_1,...,E_d\subset\er$ such that 
\begin{enumerate}
\item $E_i$ is dense in $\er$ for every $i$ 
\item for every $1\leq k\leq d$ and every $1\leq i_1<\dots<i_k\leq d$ the affine subspace defined as
$$
\bigcap_{j=1}^{k} (v_{i_j}^{\bot}+\alpha_{i_j} v_{i_j})
$$
belongs to $A_{d-k}$ whenever $\alpha_j\in E_j.$
\end{enumerate}

\end{lemma}

\begin{proof}

Define a measure $\mu$ on $(S^{d-1})^d$ as a product measure of $d$ copies of $(d-1)$-dimensional Hausdorff measure on $S^{d-1}.$
First note that the set of all linearly independent $d$-tuples $v_1,...,v_d\in S^{d-1}$ has a full measure with respect to $\mu.$
Next observe that from the definition of $\mu_i^d$ one can see that there is a set $G_i\subset G(d,i)$ of full $\nu_i^d$ measure such that for every $i$ we have that
$\H^{d-1}$-almost every translation of every $g\in G_i$ belongs to $A_i.$

Consider the set 
$$
T:=\{(v_1,...,v_d)\in (S^{d-1})^d: v_{i_1}^{\bot}\cap\dots\cap v_{i_k}^{\bot}\in G_{d-k}, 1\leq k\leq d, 1\leq i_1<\dots<i_k\leq d\}.
$$
From the representation
$$
T:=\bigcap_{k=1}^{d}\bigcap_{1\leq i_1<\dots<i_k\leq d}\{(v_1,\ldots,v_d)\in (S^{d-1})^d: v_{i_1}^{\bot}\cap\dots\cap v_{i_k}^{\bot}\in G_{d-k}\}=:T^{k}_{i_1,\ldots,i_k}.
$$
using the fact that every $T^{k}_{i_1,\ldots,i_k}$ has full measure we see that $T$ has full measure as well.
In particular, we can choose linearly independent directions $v_1,\ldots,v_d\in T$.

Now, form the definition of $T^{k}_{i_1,\ldots,i_k}$ and the definition of $G_{d-k}$ we see that we can always find a corresponding set $E^{k}_{i_1,\ldots,i_k}\subset\er^d$ of full measure such that
$$
\bigcap_{j=1}^{k} (v_{i_j}^{\bot}+\alpha_{i_j} v_{i_j})
$$
belongs to $A_{d-k}$ whenever $(\alpha_1,\ldots,\alpha_d)\in E^{k}_{i_1,\ldots,i_k}.$
Put
$$
E:=\bigcap_{k=1}^{d}\bigcap_{1\leq i_1<\dots<i_k\leq d}E^{k}_{i_1,\ldots,i_k}.
$$

To finish the proof it suffices to prove the following claim:
\begin{claim}
Suppose that $E\subset \er^d$ has full measure. Then there are $E_1,\ldots,E_d$ such that $\overline{E_1}=\dots=\overline{E_l}=\er$ and $E_1\times\dots\times E_l\subset E.$
\end{claim}
To prove the claim we will proceed by induction by $d.$
The case $d=1$ follows directly from the fact that every subset of $\er$ of full measure is dense.
Suppose now that the claim is true up to some $k$ and we need to prove it for the case $d=k+1.$
First, from the Fubini theorem we know that there is a set $Z$ in $\er$ of full measure such that for every $z\in Z$ the slice
$$
E_z=\{x\in\er^{k}:(z,x)\in E\}
$$ 
has full measure.
Since $Z$ is dense and $\er^{d-1}$ is separable we can find a countable dense set $E_1\subset Z.$
Put 
$$
E'=\bigcap_{z\in Z'} E_z,
$$
it follows that $E'$ is a set of full measure in $\er^{k}$ and $Z'\times E'\subset E.$
By induction procedure we know that there are $E_2,\ldots,E_{k+1}$ such that $\overline{E_2}=\dots=\overline{E_{k+1}}=\er$ and $E_2\times\dots\times E_{k+1}\subset E'.$
Now,  $E_1\times\dots\times E_{k+1}\subset E.$
\end{proof}

\begin{proof}[Proof of Proposition~\ref{Vitali}]
Let $v_1,\ldots,v_d$ be the unit vectors from Lemma~\ref{vektory} constructed for the sets $A_i={\mathcal A}_i^d\setminus T_i(f)$. Then, it is easy to see that
$$\left\{\bigcap_{i=1}^d (H_{v_i,\beta_i}\cap H_{-v_i,\alpha_i}):\, \alpha_i<\beta_i,\, \alpha_i,\beta_i\in E_i,\, i=1,\ldots,d\right\}$$
is a Vitali system of parallelograms not touching $f$.
\end{proof}

\begin{proposition}   \label{P_local}
If $A$ is a compact WDC subset of $\er^d$ then its normal cycle has an index function $\iota_{N_A}=:\iota_A$ fulfilling
for $\H^{d-1}$-almost all $n\in\sph$ and all $x\in\er^d$:
\begin{equation}  \label{Local_index}
\iota_{A}(x,n)=\lim_{K\to x}\esslim\limits_{\delta\to 0_+}
\left(\chi(A\cap K\cap H_{-n,-t+\delta})-\chi(A\cap K\cap H_{-n,-t-\delta})\right).
\end{equation}
\end{proposition}

\begin{proof}
Applying \eqref{E_chi}, we see that for a.a.\ $n\in S^{d-1}$, all $t\in\er$ and a.a.\ $\delta>0$,
$$\chi(A\cap H_{-n,-t+\delta})-\chi(A\cap H_{-n,-t-\delta})= 
\sum_{x:\, |x\cdot n-t|<\delta}\iota_{A}(x,n),$$
with a finite number of summands. Consequently, 
\begin{equation}
\sum_{x:\, x\cdot n=t}\iota_{A}(x,n)=\esslim_{\delta\to 0_+}\left(\chi(A\cap H_{-n,-t+\delta})-\chi(A\cap H_{-n,-t-\delta})\right).
\end{equation}
Now, fix a point $x\in\er^d$ with $x\cdot n=t$ and such that $\iota_{A}(x,n)\neq 0$. We intersect $A$ with a sufficiently small set $K\in{\mathcal V}$ and apply the same procedure as above with $A\cap K$ instead of $A$. We get 
\begin{eqnarray}  \label{ind_ball}
\lefteqn{\sum_{y:y\cdot n=t}\iota_{A\cap K}(y,n)}\\
&=&\esslim_{\delta\to 0_+}\left(\chi(A\cap K\cap H_{-n,-t+\delta})-\chi(A\cap K\cap H_{-n,-t-\delta})\right).\nonumber
\end{eqnarray}
Due to the fact that $N_{A\cap K}\llc\intr K=N_A\llc\intr K$ (Lemma~\ref{L_AK}), we obtain
$\iota_{A\cap K}(y,n)=\iota_{A}(y,n)$ for $y\in\intr K$, the sum in \eqref{ind_ball} reduces to a single summand $y=x$ if $V\ni x$ is small enough, and
$$\iota_{A}(x,n)=\esslim_{\delta\to 0_+}\left(\chi(A\cap K\cap H_{-n,-t+\delta})-\chi(A\cap K\cap H_{-n,-t-\delta})\right)$$
for sufficiently small $K\in{\mathcal V}$, which yields the desired result.
\end{proof}

\begin{theorem}  \label{T_additivity}
Let $A,B\subset\er^d$ be two compact WDC sets with nondegenerate d.c.\ auras $f,g$, respectively, and assume that $f$ and $g$ do not touch. Then $A\cup B$ and $A\cap B$ are WDC as well, and the normal cycles satisfy
$$N_A+N_B=N_{A\cap B}+N_{A\cup B}.$$
\end{theorem}

\begin{proof}
We know already that, under the given assumptions, $A\cap B$ is a WDC set, see Proposition~\ref{P_int_aur}. Clearly, $h:=\min (f,g)$ is a d.c.\ aura for $A\cup B$. We shall show that $h$ is nondegenerate. The procedure will be similar to that used in the proof of Proposition~\ref{P_int_aur}.

Choose $\ep,\delta>0$ and $U,V,W$ open neighbourhoods of $A,B,\partial A\cap\partial B$, respectively,  as in the proof of Proposition~\ref{P_int_aur}. We consider the open cover
$$A\cup B\subset G:=(U\cap\{ f<g\})\cup (V\cap\{g<f\})\cup (W\cap U\cap V).$$
Take $x\in G\setminus (A\cup B)$ and $w\in\partial^*h(x)$. If $x\in U\cap\{ f<g\}\setminus A$ then $\partial^*h(x)=\partial^*f(x)$ and, hence, $|w|\geq\ep$ by the nondegeneracy of $f$. Analogously, if $x\in V\cap\{ g<f\}\setminus B$ then $\partial^*h(x)=\partial^*g(x)$ and $|w|\geq\ep$ by the nondegeneracy of $g$. If, finally, $x\in (W\cap U\cap V)\setminus (A\cup B)$ then $w=\lambda u+(1-\lambda)v$ for some $u\in\partial^*f(x)$ and $v\in\partial^*g(x)$ and $\lambda\in[0,1]$, and we get
\begin{eqnarray*}
|\lambda u+(1-\lambda)v|&=&\sqrt{\lambda^2|u|^2+(1-\lambda)^2|v|^2+2\lambda(1-\lambda)u\cdot v}\\
&\geq&\sqrt{\lambda^2|u|^2+(1-\lambda)^2|v|^2+2\lambda(1-\lambda)(-1+\delta)|u||v|}\\
&\geq&\sqrt{+2\lambda(1-\lambda)\delta|u||v|}\\
&\geq&\sqrt{\delta/2}\ep,
\end{eqnarray*}
which shows the nondegeneracy of $h$.

It remains to verify the additivity. Applying \eqref{Tg} to the Legendrian cycles $N_A+N_B$ and $N_{A\cap B}+N_{A\cup B}$, we get
\begin{eqnarray*}
(N_A+N_B)(\phi\varphi_0)&=&\cO_{d-1}^{-1}\int_{\sph}\sum_{x\in\er^d}\phi(x,n)(\iota_A+\iota_B)(x,n)\, \H^{d-1}(dn),\\
(N_{A\cap B}+N_{A\cup B})(\phi\varphi_0)&=&\cO_{d-1}^{-1}\int_{\sph}\sum_{x\in\er^d}\phi(x,n)(\iota_{A\cap B}+\iota_{A\cup B})(x,n)\, \H^{d-1}(dn).
\end{eqnarray*}
The local form of the index function (Proposition~\ref{P_local}) and the additivity of the Euler-Poincar\'e characteristic yield the additivity of the index function:
$$\iota_A(x,n)+\iota_B(x,n)=\iota_{A\cap B}(x,n)+\iota_{A\cup B}(x,n)$$
for $\H^{d-1}$-almost all $n\in\sph$ and all $x\in\er^d$. Consequently, we have
$$(N_A+N_B)(\phi\varphi_0)=(N_{A\cap B}+N_{A\cup B})(\phi\varphi_0)$$
for all $\phi\in C^\infty_c(\er^d\times\sph)$, which implies the additivity, since any Legedrian cycle $T$ is determined by its restriction to the Gauss curvature form.
\end{proof}

\begin{remark} \rm
In fact, the local form if the index $\iota_A$ of $N_A$ is not needed for the proof of additivity. The Fu's proof of \cite[Theorem~4.2]{Fu94} (stated for subanalytic sets) could be applied instead.
\end{remark}

We show now that even locally WDC sets admit normal cycles.

\begin{theorem}\label{locth}
Any compact locally WDC set admits a normal cycle.
\end{theorem}

\begin{proof}
Let $A\subset\er^d$ be compact and locally WDC, i.e., there exists a finite cover $A\subset\bigcup_{i=1}^mU_i$ od $A$ by open sets $U_i$ and compact WDC sets $A_i$ such that $A\cap U_i=A_i\cap U_i$, $i\leq m$. Let $(g_i)$ be a smooth partition of unity such that $\spt g_i\subset U_i$, $i=1,\ldots,m$. Each set $A_i$ has a nondegenerate d.c.\ aura $f_i$ and a normal cycle $N_{A_i}=N(f_i,0)$ by Theorem~\ref{T2}. We shall show that
$$N_A:=\sum_{i=1}^mN_{A_i}\llc g_i$$
is the normal cycle of $A$.

It is easy to see that $N_A$ is Legendrian and that $\partial N_A=0$. Also, property \eqref{th} is obvious. It remains to verify \eqref{E-NC}. Applying Proposition~\ref{P_local}, we see that the index function $\iota_{A_i}(x,n)$ of $A_i$ is independent of $i$ if $x\in U_i$; let us denote it by $\iota_A(x,n)$. Note that, by construction, $\iota_A$ fulfills the local form \eqref{Local_index}.

Applying Proposition~\ref{vektory}, we can find finitely many convex compact sets $K_1,\ldots, K_k$ such that:
\begin{enumerate}
\item[(i)] $A\subset\bigcup_{i=1}^k K_i$,
\item[(ii)] any $K_i$ is contained in some $U_j$, $1\leq j\leq m$,
\item[(iii)] if $x\in\er^d$, $(x,u_i)\in\nor(f_i,0)$ and $(x,v_j)\in\nor K_j$ for all $i\in I$ and $j\in J$, where $I,J$ are some subsets of $\{1,\ldots,m\}$, $\{1,\ldots,k\}$, respectively, then all the vectors $u_i,\, i\in I, v_j,\, j\in J$, are linearly independent.
\end{enumerate}
Since each $A\cap K_i$ is WDC, we have for almost all $(v,t)\times\sph\times\er$,
$$\langle N_A,\pi_1,-v\rangle((H_{v,t}\cap K_i)\times\sph)=\chi(A\cap K_i\cap H_{v,t})$$
(cf.\ Lemma~\ref{L_AK}).
Applying additivity on both sides, we get formula \eqref{E-NC} for $A$.
\end{proof}

\section{Crofton formula} \label{S-CF}

Recall that if a compact set $A\subset\er^d$ admits a normal cycle $N_A$ then, for $k=0,1,\ldots,d-1$, the $k$th total curvature of $A$ is defined as 
$$C_k(A)=N_A(\varphi_k),$$ 
and define additionally
$$C_d(A)=\H^d(A).$$

If $E\in\A_j^d$ is an affine $j$-subspace of $\er^d$, it can be clearly identified with $\er^j$ and we can consider the notions of WDC sets and curvature measures relatively in $E$. Note that the orientations of ``gradient currents'' $[df]$, as well as of the Lipschitz-Killing differential forms $\varphi_k$ in $E$ depend on the orientation of $E$ (given e.g. by the volume form in $E$). Nevertheless, the curvature measures of a WDC subset of $E$ do not depend on the chosen orientation of $E$.

Given $v\in\sph$, we denote the following mappings:
\begin{eqnarray*}
g_v:&\, (x,n)\mapsto x\cdot v,\quad& (x,n)\in\er^d\times\sph,\\
h_v:&\, (x,n)\mapsto (x,p_{v^\perp}n), \quad&  (x,n)\in\er^d\times\sph,\\
f_v:&\, (x,n)\mapsto \nu(h_v(x,n)), \quad&  (x,n)\in\er^d\times(\sph\setminus\{-v,v\})
\end{eqnarray*}
(recall that $\nu(x,n)=(x,n/|n|)$).

With a pair $(v,t)\in\sph\times\er$, we associate the affine $(d-1)$-subspace
$$E_{v,t}=\{x\in\er^d:\, x\cdot v=t\}.$$
We assign an orientation to $E_{v,t}$ by the volume $(d-1)$-form $v\lrc\Omega_d$.
We say that a nondegenerate aura $f$ in $\er^d$ touches $E_{v,t}$ if it touches one of the two halfspaces with boundary $E_{v,t}$.

\begin{lemma}  \label{CL1}
Let $f$ be a nondegenerate d.c.\ aura in $\er^d$. Then, for $\H^d$-almost all $(v,t)\in\sph\times\er$, the restriction $f|E_{v,t}$ is a nondegenerate d.c.\ aura in $E_{v,t}$ and 
$$N(f|E_{v,t},0)=(-1)^d(f_v)_{\#}\langle N(f,0),g_v,t\rangle.$$
\end{lemma}

\begin{proof}
We know from Proposition~\ref{P_th} that $f$ does not touch $E_{v,t}$ for $\H^d$-almost all $(v,t)$. If this is the case and $\nabla f(x)$ exists at some $x\in E_{v,t}$ then $\nabla (f|E_{v,t})(x)=p_{v^\perp}\nabla f(x)$ and it is not difficult to verify from the definition the relation
$$[d(f|E_{v,t})]=(h_v)_{\#}\langle [df],g_v,t\rangle.$$
Let $U$ be an open neighbourhood of $A=f^{-1}\{0\}$ in $\er^d$ as in Definition~\ref{aura}. If $f$ does not touch $E_{v,t}$ then $U\cap E_{v,t}$ is an open neighbourhood of $(f|E_{v,t})^{-1}\{0\}$ guaranteeing the nondegeneracy of $f|E_{v,t}$ and we have 
\begin{eqnarray*}
N(f|E_{v,t},0)
&=&\nu_{\#}\Big( -\partial([d(f|E_{v,t})]\llc\pi_0^{-1}(U\setminus A))\llc\pi_0^{-1}(U)\Big)\\
&=&-\nu_{\#}\Big( \partial((h_v)_{\#}\langle [df],g_v,t\rangle\llc\pi_0^{-1}(U\setminus A))\llc\pi_0^{-1}(U)\Big)\\
&=&-(\nu\circ h_v)_{\#}\Big(\partial\langle [df],g_v,t\rangle\llc\pi_0^{-1}(U\setminus A)\Big)\llc\pi_0^{-1}(U)\\
&=&-(f_v)_{\#}\Big((-1)^d\langle\partial([df]\llc\pi_0^{-1}(U\setminus A)),g_v,t\rangle\llc\pi_0^{-1}(U)\Big)\\
&=&(-1)^d(f_v)_{\#}\nu_{\#}(-\langle\partial([df]\llc\pi_0^{-1}(U\setminus A))\llc\pi_0^{-1}(U),g_v,t\rangle)\\
&=&(-1)^d(f_v)_{\#}\langle N(f,0),g_v,t\rangle,
\end{eqnarray*}
which proves the assertion. We have used the basic properties of slices from \cite[\S4.3]{Fe69}.
\end{proof}

Let $\varphi_k^{(v)}$ be the $k$th Lipschitz-Killing differential form in $E_{v,t}$.

\begin{lemma}  \label{CL2}
For any $(x,n)\in\er^d\times\sph$ and $k=0,\ldots,d-2$,
$$\int_{\sph}f_v^{\#}(g_v^{\#}\Omega_1\wedge\varphi_k^{(v)})(x,n)\, \H^{d-1}(dv)=(-1)^d \frac{2\pi^{d/2}\Gamma\left(\frac {k+2}2\right)}{\Gamma\left(\frac{k+1}2\right)\Gamma\left(\frac{d+1}2\right)}\varphi_{k+1}(x,n).$$
\end{lemma}

\begin{proof}
Let $\{a_1,\ldots,a_d\}$ be a positively oriented orthonormal basis of $\er^d$ with $a_d=n$. The vectors
$$a_{I,J}:=\bigwedge_{i\in I}(a_i,0)\wedge\bigwedge_{j\in J}(0,a_j),\quad I,J\subset\{1,\ldots,d\},\, |I|+|J|=d-1,$$
form a basis of $\bigwedge_{d-1}\er^{2d}$ ($|I|,|J|$ denote the cardinality of $I,J$, respectively). It is thus sufficient to verify the equality of the two forms on the basis vectors.

Let $I,J\subset\{1,\ldots.d\}$ with $|I|+|J|=d-1$ be given and denote
$$\sigma_{I,J}:=\left\langle\bigwedge_{i\in I}a_i\wedge\bigwedge_{j\in J}a_i\wedge n\right\rangle\in\{-1,0,1\}.$$
We have
\begin{eqnarray*}
I(v)&:=&\left\langle a_{I,J},f_v^{\#}(g_v^{\#}\Omega_1\wedge\varphi_k^{(v)})(x,n)\right\rangle\\
&=&\left\langle\bigwedge_{i\in I}Df_v(x,n)(a_i,0)\wedge\bigwedge_{j\in J}Df_v(x,n)(0,a_j),(g_v^{\#}\Omega_1\wedge\varphi_k^{(v)})(f_v(x,n))\right\rangle.
\end{eqnarray*}
An elementary calculation yields that $Dg_v(x,n)(a,b)=a\cdot v$ and $Df_v(x,n)(a,b)=(a,\tilde{b})$ with $\tilde{b}=p_{v^\perp\cap n^\perp}b/|p_{v^\perp}n|$. Thus
\begin{eqnarray*}
I(v)&=&\left\langle \bigwedge_{i\in I}(a_i,0)\wedge\bigwedge_{j\in J}(0,\tilde{a}_j),(g_v^{\#}\Omega_1\wedge\varphi_k^{(v)})(f_v(x,n))\right\rangle\\
&=&\sum_{i\in I}(-1)^{\sigma_I(i)-1}(a_i\cdot v)
\left\langle\bigwedge_{j\in I,j\neq i}(a_i,0)\wedge\bigwedge_{j\in J}(0,\tilde{a}_j),\varphi_k^{(v)}(f_v(x,n))\right\rangle,
\end{eqnarray*}
where $\sigma_I(i)$ is the order number of $i$ in $I$.
From the definition of the differential form $\varphi_k^{(v)}$, the last expression vanishes unless $|I|=k+1$. Assume that this is the case; then we have
\begin{eqnarray*}
I(v)&=&\cO_{d-2-k}^{-1}\sum_{i\in I}(-1)^{\sigma_I(i)-1}(a_i\cdot v)
\left\langle\bigwedge_{j\in I,j\neq i}a_i\wedge\bigwedge_{j\in J}\tilde{a}_j\wedge \frac{p_{v^\perp}n}{|p_{v^\perp}n|}, v\lrc\Omega_d\right\rangle\\
&=&\cO_{d-2-k}^{-1}\sum_{i\in I}(-1)^{\sigma_I(i)-1}(a_i\cdot v) |p_{v^\perp}n|^{k+1-d}\\
&& \hspace{3cm}\times
\left\langle\bigwedge_{j\in I,j\neq i}a_i\wedge\bigwedge_{j\in J}p_{v^\perp\cap n^\perp}a_j\wedge p_{v^\perp}n\wedge v, \Omega_d\right\rangle\\
&=&\cO_{d-2-k}^{-1}\sum_{i\in I}(-1)^{\sigma_I(i)-1}(a_i\cdot v) |p_{v^\perp}n|^{k+1-d}\\
&& \hspace{3cm}\times\left\langle\bigwedge_{j\in I,j\neq i}a_i\wedge\bigwedge_{j\in J}a_j\wedge n\wedge v, \Omega_d\right\rangle\\
&=&\cO_{d-2-k}^{-1}\sum_{i\in I}(-1)^{\sigma_I(i)-1}(a_i\cdot v) |p_{v^\perp}n|^{k+1-d}
(-1)^{d-\sigma_I(i)-1}(a_i\cdot v)\sigma_{I,J}\\
&=&\cO_{d-2-k}(-1)^d(k+1)|p_{v^\perp}n|^{k+1-d} (a_i\cdot v)^2 \sigma_{I,J}.
\end{eqnarray*}
A routine calculation verifies that
$$\int_{\sph} |p_{v^\perp}n|^{k+1-d} (a_i\cdot v)^2\, \H^{d-1}(dv)=\frac{\pi^{d/2}\Gamma\left(\frac {k+2}2\right)}{\Gamma\left(\frac{k+3}2\right)\Gamma\left(\frac{d+1}2\right)}.$$
Clearly, also
$\langle a_{I,J},\varphi_{k+1}(x,n)\rangle =\cO_{d-2-k}^{-1}\sigma_{I,J}$ 
if $|I|=k+1$ and $0$ otherwise, which completes the proof.
\end{proof}

\begin{proof}[Proof of Theorem~\ref{T3}]
Let $f$ be a nondegenerate d.c.\ aura with $f^{-1}\{0\}=A$.
If $k=m$ then the formula follows easily from the Fubini theorem.
If $k<m=d-1$ then we have using Lemma~\ref{CL1} and Lemma~\ref{CL2},
\begin{eqnarray*}
\lefteqn{\int_{\A^d_{d-1}}C_k(A\cap E)\,\mu^d_{d-1}(dE)}\\
&=&\cO_{d-1}^{-1}\int_{\sph}\int_{\er}N(f|E_{v,t})(\varphi^{(v)}_k)\, dt\,\H^{d-1}(dv)\\
&=&\cO_{d-1}^{-1}(-1)^d\int_{\sph}\int_{\er}(f_v)_{\#}\langle N(f,0),g_v,t\rangle (\varphi^{(v)}_k)\, dt\,\H^{d-1}(dv)\\
&=&\cO_{d-1}^{-1}(-1)^d\int_{\sph}N(f,0)(f_v^{\#}(g_v^{\#}\Omega_1\wedge\varphi^{(v)}_{k}))\,\H^{d-1}(dv)\\
&=&\beta^d_{k+1,d-1}N(f,0)(\varphi_{k+1})\\
&=&\beta^d_{k+1,d-1}C_{k+1}(A).
\end{eqnarray*}
For general $m<d$ we use induction on $m$ and relation \eqref{int_geom}.
\end{proof}

\section{Open problems}

\begin{problem}
{\rm Is the unit normal bundle of a nondegenerate d.c.\ aura rectifiable? In particular, does the principal kinematic formula hold for pairs of WDC sets?
Note that the first part of the question would also solve a long time open problem of the rectifiability of the set of directions of line segments contained on the boundary of a convex body. }
\end{problem}

\begin{problem}\label{locwdc}
{\rm Is there a natural definition of the normal bundle of WDC sets not depending on the corresponding aura? Is, for instance, the sum of two nondegenerated d.c.\ auras of one WDC set again a nondegenerated aura of that set? Are the classes of compact locally WDC sets and compact WDC sets identical?}
\end{problem}

\begin{problem}
{\rm What can be said about relationship between WDC sets and other classes of sets admitting the normal cycle? For instance, what is the relationship between d.c.\ domains and the Lipschitz submanifolds with bounded curvature investigated in \cite{RZ05}?}
\end{problem}

\section*{acknowledgement}
The first version of this paper contained a weaker result (concerning d.c.\ domains) proved by technically more complicated means. We are deeply grateful to Joseph Fu who suggested us to extend our result to WDC sets and to use the technique of auras. This surely helped to improve the quality of the paper significantly.


\begin{thebibliography}{99}
\bibitem{AA} Brothers, J.E., Some open problems, in: {\it Geometric Measure Theory and the Calculus of Variations},
Almgren, F.J. and Allard, W. K. eds., Proc. Sym. Pure Math. {\bf 44} 1986, pp. 441--464
\bibitem{ELR} Ewald, G., Larman, D.G., Rogers, C.A.: The directions of the line segments and of the r-dimensional balls on the boundary of a convex body in Euclidean space. {\it  Mathematika} {\bf 17}, (1970), 1–-20
\bibitem{Fe59} Federer, H.: Curvature measures. {\it Trans.\ Amer.\ Math.\ Soc.} {\bf 93} (1959), 418--491
\bibitem{Fe69} Federer, H.: {\it Geometric Measure Theory}. Springer, Berlin 1969
\bibitem{Fu1} Fu, J.H.G.: Monge-Amp\`ere functions I. {\it Indiana Univ. Math. J.} {\bf 38}, no. 3, (1989), 745--771
\bibitem{Fu94} Fu, J.H.G.: Curvature measures of subanalytic sets.
    {\it Amer.\ J.\ Math.} {\bf 116} (1994), 819--880
\bibitem{Fu00} Fu, J.H.G.: Stably embedded surfaces of bounded integral curvature. {\it Adv.\ Math.} {\bf 152} (2000), no. 1, 28--71
\bibitem{Fu2} Fu, J.H.G.: An extension of Alexandrov's theorem on second derivatives of convex functions. {\it Adv. Math.} {\bf 228} (2011), 2258--2267
\bibitem{Har} Hartman, P.: On functions representable as a difference of convex functions. {\it Pacific J. Math.} {\bf 9} (1959), 707–-713
\bibitem{Jerr} Jerrard, R.L.: Some rigidity results related to Monge–Ampère functions. {\it Canad.\ J. Math.} {\bf 62} (2010), 320--345
\bibitem{Jor} Jordan, Ch.: {\it Calculus of finite differences. Third Edition.} Chelsea Publishing Co., New York 1965
\bibitem{Kl} Kleinjohann, N.: N\"achste Punkte in der Riemannschen Geometrie. {\it Math.\ Z.} {\bf 176} (1981), 327--344
\bibitem{KP08} Krantz, S.G., Parks, H.R.: {\it Geometric Integration Theory.} Birkh\"auser, Boston 2008
\bibitem{PZ} Pavlica, D., Zaj\' i\v cek, L.: On the directions of segments and $r$-dimensional balls on a convex surface. {\it  J. Convex Anal.} {\bf 14}, no. 1, (2007), 149--167
\bibitem{RZ01} Rataj, J., Z\"ahle, M.: Curvatures and currents for unions of sets with positive reach, II. {\it Ann.\ Global Anal.\ Geom.} {\bf 20} (2001), 1--21
\bibitem{RZ05} Rataj, J., Z\"ahle, M.: General normal cycles and Lipschitz manifolds of bounded curvature. {\it Ann.\ Global\ Anal.\ Geom.} {\bf 27} (2005), 135--156
\bibitem{RZaj} Rataj, J., Zaj\'\i\v cek, L.: Critical values and level sets of distance functions in Riemannian, Alexandrov and Minkowski spaces. {\it Houston J. Math.}, no. 2 (2012), 445--467
\bibitem{SnWe} Schneider, R.; Weil, W., {\it Stochastic and integral geometry. Probability and its Applications}. Springer-Verlag, Berlin 2008
\bibitem{SW} Sulanke, R. Wintgen, P.: {\it Differentialgeometrie und Faserb\"undel}. Berlin 1972.
\bibitem{VZ} Vesel\'y, L., Zaj\'\i\v cek, L.: On compositions of d.c. functions and mappings. {\it J. Convex Anal.}, {\bf 16} (2009), no. 2, 423--439
\bibitem{Z86} Z\"{a}hle, M.: Integral and current representation of Federer's curvature measures. {\it Arch.\ Math.} {\bf 46} (1986), 557--567
\bibitem{Z87} Z\"{a}hle, M.: Curvatures and currents for unions of sets with positive reach. {\it Geom.\ Dedicata} {\bf 23} (1987), 155--171 







\end{thebibliography}
\end{document}